\documentclass[12pt]{article}

\usepackage{amsmath}
\allowdisplaybreaks[4]

\usepackage[all]{xy}
\usepackage{amssymb}
\usepackage{amsthm}
\usepackage{hyperref}
\hypersetup{colorlinks=true,linkcolor=blue,citecolor=red}
\usepackage{amsmath}
\usepackage{amscd}
\usepackage{verbatim}
\usepackage{eurosym}
\usepackage{float}
\usepackage{color}
\usepackage{dcolumn}
\usepackage[mathscr]{eucal}
\usepackage[all]{xy}
\usepackage{hyperref}
\usepackage{mathrsfs}
\usepackage{amsmath}
\usepackage{amssymb}
\usepackage{amsfonts,ifpdf}
\usepackage{graphicx}
\usepackage{times}
\usepackage{float}
\usepackage{epstopdf}
\usepackage{cite}
\usepackage{youngtab}
\usepackage{ytableau}
\ytableausetup
{mathmode, boxsize=0.9em}

\newlength{\boxedparwidth}
  {\begin{center} \begin{tabular}{|@{\hspace{.315in}}c@{\hspace{.15in}}|}
                  \hline \\ \begin{minipage}[t]{\boxedparwidth}
                  \setlength{\parindent}{.25in}}%
  {\end{minipage} \\ \\ \hline \end{tabular} \end{center}}

\allowdisplaybreaks

\newtheorem{thm}{Theorem}[section]

\newtheorem{lem}[thm]{Lemma}
\newtheorem{cor}[thm]{Corollary}

\numberwithin{equation}{section}

\begin{document}

\begin{center}

 {\Large \bf Asymptotics and inequalities for the broken $k$-diamond
 partition function}
\end{center}
\vskip 3mm

\begin{center}
Ying Zhong\\[8pt]
School of Mathematics\\
Hunan University\\
Changsha 410082, P. R. China
\\[15pt]

Emails: YingZhong@hnu.edu.cn
\\[15pt]

\end{center}

\vskip 3mm

\begin{abstract}
Many papers have studied inequalities for Andrews and Paule's broken $k$-diamond partition function $\Delta_{k}(n)$ when $k=1$ or $2$. In this paper, we derive an exact formula for $\Delta_{k}(n)$ when $k\geq 1$. Building on this result, we also derive an asymptotic formula for $\Delta_{k}(n)$ with an explicit error bound. Using this formula, we prove that for $k\geq 1$ and sufficiently large $n$, $\Delta_{k}(n)$ satisfies the Tur\'an and Laguerre inequalities of any order and exhibits asymptotic complete monotonicity. Define $n_k:=\max\left\{\left\lceil
8k^{3}+\frac{k+1}{12}\right\rceil,526\right\}$. Furthermore, we show that $\Delta_{k}(n)$ is log-concave for $k\ge3$ and $
n\ge n_k
$. Consequently, it follows that $\Delta_{k}(a)\Delta_{k}(b)\ge\Delta_{k}(a+b)$ for $k\ge3$ and $a,b \ge n_k$.

\vskip 6pt

\noindent {\bf AMS Classifications}:11P82, 05A19, 30A10
\\ [7pt]
{\bf Keywords}: broken $k$-diamond partition function, asymptotics, Tur\'an inequalities, Laguerre inequalities, asymptotically complete monotonicity
\end{abstract}

\section{Introduction}

The objective of this paper is to establish asymptotics for the broken $k$-diamond partition function and further show that this function satisfies the Tur\'an inequalities, the Laguerre inequalities, and asymptotically complete monotonicity, when $k\in\mathbb{N}$.

The broken $k$-diamond partitions were introduced by Andrews and Paule \cite{Andrews-Paule-2007}.
Let $\Delta_{k}(n)$ denote the number of broken $k$-diamond partitions of $n$. Andrews and Paule \cite{Andrews-Paule-2007} gave the following generating function for $\Delta_{k}(n)$:
\begin{align}\label{genef}
\sum_{n=0}^{\infty} \Delta_{k}(n) q^{n} &=\prod_{n=1}^{\infty} \frac{\left(1-q^{2 n}\right)\left(1-q^{(2 k+1) n}\right)}{\left(1-q^{n}\right)^{3}\left(1-q^{(4 k+2) n}\right)}.
\end{align}
The broken $k$-diamond partition function $\Delta_{k}(n)$ is known for its many elegant arithmetic properties. In particular, numerous Ramanujan-like congruences satisfied by $\Delta_k(n)$ were proved by many authors, including Andrews and Paule \cite{Andrews-Paule-2007}, Chan \cite{Chan-2008}, Chen, Fan and Yu \cite{Chen-Fan-Yu-2014}, Hirschhorn \cite{Hirschhorn-2018}, Paule and Radu \cite{Paule-Radu-2010} and so on.

A sequence $\{\alpha(n)\}_{n\geq 0}$ of real numbers satisfies the (second order) Tur\'an inequalities or to be log-concave if for $n\geq 1$,
\[\alpha^2(n)\geq \alpha(n-1)\alpha(n+1).\]
As stated by Chen, Jia and Wang \cite{Chen-Jia-Wang-2019} and Griffin, Ono, Rolen and Zagier \cite{Griffin-Ono-Rolen-Zagier-2019}, the higher order Tur\'an inequalities are conveniently formulated in terms of the Jensen polynomials. The Jensen polynomials $J_\alpha^{d, n}(X)$ of degree $d$ and shift $n$ associated to the sequence $\left\{\alpha(n)\right\}_{n \geq 0}$ are defined by
$$
J_\alpha^{d, n}(X):=\sum_{j=0}^d\binom{d}{j} \alpha(n+j) X^j.
$$
We say that the sequence $\left\{\alpha(n)\right\}_{n \geq 0}$ satisfies the order $d$ Tur\'an inequality at $n$ if and only if $J_\alpha^{d, n-1}(X)$ is hyperbolic.
Additionally, A sequence $\{\alpha(n)\}_{n\geq 0}$ is said to satisfy the Laguerre inequalities of order $m$ if for all $n\geq 0$,
\begin{equation*}
L_m(\alpha(n)):=\frac 12\sum_{k=0}^{2m}(-1)^{k+m} {{2m}\choose k}\alpha(n+k)\alpha(2m-k+n) \geq 0.
\end{equation*}
The Tur\'an inequalities and the Laguerre inequalities arise from the study of the Laguerre--P\'{o}lya class, which are closely related to the Riemann hypothesis. We refer the interested readers to  \cite{Csordas-Varga-1990,Dimitrov-1998,Dimitrov-Lucas-2011,
	Schur-Polya-1914,Szego-1948}. Consequently, this area has received widespread attention. Many researchers have been investigating the Tur\'{a}n inequalities \cite{Chen-Jia-Wang-2019,DeSalvo-Pak-2015,GORT-2019,Griffin-Ono-Rolen-Zagier-2019,
Ono-Pujahari-Rolen,Pandey-2024} and the Laguerre inequalities \cite{Wang-Yang-2022,Wang-Yang-2024,Yang-2024} for different
kinds of sequences.

Define
$$
\Delta^r \left(\alpha(n)\right):=\sum_{k=0}^r\binom{r}{k}(-1)^{r+k} \alpha(n+k) .
$$
We say that the sequence $\left\{\alpha(n)\right\}_{n \geq 0}$ satisfies $r$-order monotonicity if the sign of $\Delta^r \left(\alpha(n)\right)$ remains invariant. If $\left\{\alpha(n)\right\}_{n \geq 0}$ satisfies $r$-order monotonicity for any $r$, we say that it possesses complete monotonicity. Moreover, $\left\{\alpha(n)\right\}_{n \geq 0}$ is said to satisfy asymptotically complete monotonicity if, for any $r$, there exists an integer $N(r)$ such that $\left\{\alpha(n)\right\}_{n \geq 0}$ satisfies $r$-order monotonicity for $n>N(r)$. It should be noted that the complete monotonicity is another property related to the Laguerre--P\'{o}lya class, see \cite{Craven-Csordas-1983}.  Good \cite{Good-1997} conjectured and Gupta \cite{Gupta-1978} later proved, that for each $r\geq1$, $\Delta^r \left(p(n)\right) > 0$ holds for sufficiently large $n$, where $p(n)$ denotes the ordinary partition function. Furthermore, Wang, Xie and Zhang \cite{Wang-Xie-Zhang-2018} established the asymptotically complete monotonicity for the overpartition function.

Recently, inequalities for the broken $k$-diamond partition function have also received significant attention from many scholars.
Dong, Ji and Jia \cite{Dong-Ji-Jia-2023} obtained an asymptotic formula for $\Delta_{k}(n)$, where $k=1$ or $2$. Based on this asymptotic formula, they showed that $\Delta_{k}(n)$ satisfies the Tur\'an inequalities of any order for sufficiently large $n$ and is log-concave for $n\geq1$. They also conjectured that these results still hold for $k\geq 3$. Applying the asymptotic formula given in \cite{Dong-Ji-Jia-2023}, when $k=1$ or $2$, Jia \cite{Jia-2023} proved that $\Delta_{k}(n)$ satisfies the third order Tur\'an inequalities for $n\geq 6$ and Yang \cite{Yang-2024} established the Laguerre inequalities of order $2$ and determinantal inequalities of order $3$ for $\Delta_{k}(n)$.

However, the existing asymptotics and inequalities of $\Delta_{k}(n)$ only hold for $k=1$ or $2$. In this paper, we intend to extend these results to any positive integer $k$ and thus resolve the conjectures proposed by Dong, Ji and Jia \cite{Dong-Ji-Jia-2023}.

To state our asymptotic formula for $\Delta_{k}(n)$, we need to introduce some definitions. Here and throughout, $j$ and $h$ are positive integers with $\gcd(h, j)= 1$, and $h'\in\mathbb{Z}$ is defined by $hh' \equiv -1 \pmod j$. Define $0 \leq h_{j, t}<\frac{j}{\gcd(j, t)}$ to satisfy the congruence $h \frac{t}{\gcd(j, t)} h_{j, t} \equiv-1\pmod {\frac{j}{\gcd(j, t)}}$. We define
\begin{equation}\label{def:Omega}
\Omega_{k}(h,j):=\frac{\omega_{h,j}^3\omega_{h\frac{4k+2}{\gcd(j, 4k+2)},\frac{j}{\gcd(j, 4k+2)}}}{\omega_{h\frac{2}{\gcd(j, 2)},\frac{j}{\gcd(j, 2)}}\omega_{h\frac{2k+1}{\gcd(j, 2k+1)},\frac{j}{\gcd(j, 2k+1)}}},
\end{equation}
where
$$
\omega_{h, j}:=\exp \left(\pi i\sum_{\mu \pmod j}\left(\left(\frac{\mu}{j}\right)\right)\left(\left(\frac{h \mu}{j}\right)\right)\right),
$$
with the sawtooth function defined as
$$
((x)):= \begin{cases}x-\lfloor x\rfloor-\frac{1}{2}, & \text { if } x \in \mathbb{R} \backslash \mathbb{Z},
\\ 0, & \text { if } x \in \mathbb{Z} .\end{cases}
$$
It is worth noting that $\omega_{h, j}$ comes from the $\eta$-multiplier and can alternatively be written in terms of explicit roots of unity, see, for example, \cite[(5.2.4)]{Andrews-1976}.

We also define
\begin{align}\label{def:Aj}
A_{j}(n,n_1,n_2,n_3,n_4):=\sum_{\substack{0 \leq h < j \\ \gcd(h,j) = 1}}\Omega_{k}(h,j) e^{-\frac{2\pi i h}{j}\left(n-\frac{k+ 1}{12}\right)}
e^{\frac{\pi i}{j}b_{h,j}(n_1,n_2,n_3,n_4)},
\end{align}
where
\begin{align}\label{def:b}
b_{h,j}(n_1,n_2,n_3,n_4):=2n_1 h'+2n_2\gcd(j,4k+2)h_{j,4k+2}+n_3(3n_3-1)\gcd(j,2)h_{j,2}
\notag\\
+n_4(3n_4-1)\gcd(j,2k+1)h_{j,2k+1}.
\end{align}
Moreover, let
\begin{align}\label{def:r}
&r_{h,j}(n_1,n_2,n_3,n_4)
:=\frac{\left(\gcd(j,2)^2-2\right)\left(\gcd(j,2k+1)^2
-(2k+1)\right)+2(4k+2)}{4k+2}
\\\notag
&-24\left(n_1+\frac{\gcd(j,4k+2)^2}{4k+2}n_2
+\frac{\gcd(j,2)^2}{4}n_3(3n_3-1)
+\frac{\gcd(j,2k+1)^2}{4k+2}n_4(3n_4-1)\right).
\end{align}

\begin{thm}\label{exact}
Let $I_{2}(s)$ be the modified Bessel function of the first kind of order $2$ and \begin{equation}\label{defi-xk}
x_{k}(n):=\frac{\pi\sqrt{24n-(2k+2)}}{6}.
\end{equation}
For $k\geq 1$ and $n-\frac{k+1}{12}>0$, we have
\begin{align*}
\Delta_{k}(n)&=\frac{ \pi^3}{18 x^2_{k}(n)}\sum_{j=1}^{\infty} \frac{1}{j}
\sum_{\substack{n_1,n_2\geq 0 \\ n_3,n_4\in \mathbb{Z} \\r_{h,j}(n_1,n_2,n_3,n_4)\geq 0}}A_{j}(n,n_1,n_2,n_3,n_4)p_3(n_1)p(n_2)(-1)^{n_3+n_4}
\\
&\quad \times r_{h,j}(n_1,n_2,n_3,n_4) I_{2} \left( \frac{x_{k}(n)\sqrt{r_{h,j}(n_1,n_2,n_3,n_4)}}{j}  \right),
\end{align*}
where $p_k(n)$ denotes the number of $k$-colored partitions of $n$.
\end{thm}
According to the exact formula stated in Theorem \ref{exact}, by extracting the main term and bounding the error term, we derive the following asymptotic formula of $\Delta_{k}(n)$.

\begin{thm}\label{asympform}
Let $\alpha_k:=\frac{5k+2}{2k+1}$ and
$\beta_k:=\frac{5k-10}{2k+1}$.
For $k\geq 3$ and $n-\frac{k+1}{12}>0$, we have
\begin{equation*}
\Delta_{k}(n)=M_k(n) +R_{k}(n),
\end{equation*}
where
\begin{align}\label{eq:M-R}
M_k(n):=\frac{ \pi^3\alpha_k}{18 {x^2_{k}(n)}} I_2\left(\sqrt{\alpha_k } {x_{k}(n)}\right),
~~
|R_{k}(n)|\leq
12k^3 e^{\pi\sqrt{k/3}}x_{k}^{-\frac{3}{2}}(n)
\exp\left( \sqrt{\beta_k}x_{k}(n) \right).
\end{align}
\end{thm}
In view of Theorem \ref{asympform} and  \cite[Theorem 1.2]{Dong-Ji-Jia-2023}, we prove the following two theorems by using the general results of Griffin, Ono, Rolen and Zagier \cite{Griffin-Ono-Rolen-Zagier-2019}, and Wang and Yang \cite{Wang-Yang-2024}.
\begin{thm}\label{thm:ineq-any-order}
For $k\ge1$ and sufficiently large $n$, $\Delta_k(n)$ satisfies the Tur\'an and Laguerre inequalities of any order.
\end{thm}

\begin{thm}\label{thm:ineq-any-order-d}
For $k\ge1$ and sufficiently large $n$, we have
\begin{equation*}
\Delta^r \left(\Delta_k(n)\right)>0.
\end{equation*}
\end{thm}

We also determine the threshold
for the (second order) Tur\'an inequality with the aid of Theorem \ref{asympform}.
\begin{thm}\label{thm:main}
For $k\ge3$ and $
n\ge \max\left\{\left\lceil
8k^{3}+\frac{k+1}{12}\right\rceil,526\right\}
$, the (second order) Tur\'an inequality holds:
\[
\Delta_k^2(n)\ \ge\ \Delta_k(n-1)\,\Delta_k(n+1).
\]
\end{thm}

As noted by Asai, Kubo and Kuo \cite{Asai-Kubo-Kuo-00} and Sagan \cite{Sagan-1988}, we see that Theorem \ref{thm:main} implies the following multiplicative properties of $\Delta_{k}(n)$.

\begin{cor}\label{thm-Delta-2}
For $k\geq 3$ and   $a,\,b\ge \max\left\{\left\lceil
8k^{3}+\frac{k+1}{12}\right\rceil,526\right\}
$,
\begin{equation*}
\Delta_{k}(a)\Delta_{k}(b)\ge\Delta_{k}(a+b).
\end{equation*}
\end{cor}
It is worth noting that the multiplicative properties were first studied by Bessenrodt and Ono \cite{Bessenrodt-Ono-2016} involving the ordinary partition function $p(n)$. Since then, various partition functions have been investigated in that direction, for instance,
Beckwith and Bessenrodt  \cite{Beckwith-Bessenrodt-2016} proved the multiplicative properties of  $k$-regular partition function  and   Bringmann, Kane, Rolen and Tripp \cite{Bringmann-Kane-Rolen-Tripp-2021}  obtained the multiplicative properties of $k$-colored partition function, which resolved a conjecture of Chern, Fu and Tang \cite{Chern-Fu-Tang-2018}.

This article is organized as follows. In Section \ref{sec:asm}, we obtain an exact formula for $\Delta_{k}(n)$ by using Zuckerman's formula for the Fourier coefficients of weakly holomorphic modular forms of arbitrary nonpositive weight, and subsequently prove Theorem \ref{asympform}. This theorem allows us to prove in Section \ref{sec:ineq-any} that $\Delta_{k}(n)$ satisfies the Tur\'an and Laguerre inequalities of any order and asymptotically complete monotonicity for sufficiently large $n$. In Section \ref{sec:log}, we further establish the log-concavity of $\Delta_{k}(n)$.

\section{Asymptotic formula for $\Delta_{k}(n)$}\label{sec:asm}
In this section, we present an exact formula for $\Delta_{k}(n)$ by a powerful result of Zuckerman \cite{Zuckerman-1939}. Using this formula, we further obtain an asymptotic formula for $\Delta_{k}(n)$ with an explicit error bound. Specifically, Zuckerman applied the Circle Method to furnish exact formulas for the Fourier coefficients of weakly holomorphic modular forms of arbitrary nonpositive weight on finite index subgroups of $\mathrm{SL}_{2}(\mathbb{Z})$ in terms of the cusps of the underlying subgroup and the principal parts of the form at each cusp. The relevant result from \cite[Theorem 1]{Zuckerman-1939} can be stated as follows.

\begin{thm}\label{zuc-asym}
Assume that $F$ is a weakly holomorphic modular form of weight $\kappa \leq 0$ on a congruence subgroup
with transformation law
\[
F\left( \frac{h}{j} + \frac{iz}{j} \right) = \chi(\gamma_{h,j}) (-iz)^{-\kappa} F\left( \frac{h^{\prime}}{j} + \frac{i}{jz} \right)
\]
for some multiplier $\chi:\mathrm{SL}_{2}(\mathbb{Z})\rightarrow\mathbb{C}$
and where $\gamma_{h,j} := \begin{pmatrix} h & \beta \\ j & -h^{\prime} \end{pmatrix} \in \mathrm{SL}_{2}(\mathbb{Z})$.
Suppose that $F$ has the
Fourier expansion at $i\infty$ given by $F(\tau) = \sum_{n \gg -\infty} a(n) q^{n+\lambda}$ and
Fourier expansions at each rational number $0 \leq \frac{h}{j} < 1$ given by
$\left. F \right|_{\kappa} \gamma_{h,j}(\tau) = \sum_{n \gg -\infty} a_{h,j}(n) q^{\frac{n + \lambda_{h,j}}{c_{j}}}
$. Then for $n + \lambda > 0$, we have
\begin{align*}
a(n) &= 2\pi (n + \lambda)^{\frac{\kappa - 1}{2}} \sum_{j=1}^{\infty} \frac{1}{j} \sum_{\substack{0 \leq h < j \\ \gcd(h,j) = 1}} \chi(\gamma_{h,j}) e^{-\frac{2\pi i (n + \lambda) h}{j}} \\
&\quad \times \sum_{m + \lambda_{h,j} \leq 0} a_{h,j}(m) e^{\frac{2\pi i}{j c_{j}} (m + \lambda_{h,j}) h^{\prime}} \left( \frac{|m + \lambda_{h,j}|}{c_{j}} \right)^{\frac{1 - \kappa}{2}} I_{-\kappa + 1} \left( \frac{4\pi}{j} \sqrt{\frac{(n + \lambda) |m + \lambda_{h,j}|}{c_{j}}} \right).
\end{align*}
\end{thm}

\noindent{\it Proof of Theorem \ref{exact}.}
Combining with \eqref{genef}, we are led to
\begin{align}\label{dia-eta}
\sum_{n=0}^{\infty} \Delta_{k}(n) q^{n} =q^{\frac{k+1}{12}} f_{k}(\tau)
\end{align}
with
\begin{equation*}
f_{k}(\tau):=\frac{\eta\left(2\tau\right)\eta\left((2 k+1)\tau \right)}{\eta\left(\tau\right)^{3}\eta\left((4 k+2)\tau\right)},
\end{equation*}
where $\eta(\tau)$ is the Dedekind eta function defined on the upper half-plane $\mathbb{H}=\{\tau \in \mathbb{C} | \operatorname{Im} \tau>0\}$ by
\begin{equation}\label{eta-def}
\eta(\tau):=q^{\frac{1}{24}}\prod_{n=1}^\infty(1-q^n),~q:=e^{2\pi i\tau}.
\end{equation}
It should be noted that $f_{k}$ is a weakly holomorphic modular form of weight $-1$ on a congruence subgroup $\Gamma_0(4 k+2)$.
We now establish the asymptotics of $\Delta_{k}(n)$ by applying Theorem \ref{zuc-asym} to $f_{k}$.

Equation \eqref{dia-eta} implies that
\begin{equation*}
f_{k}(\tau)=\sum_{n=0}^{\infty} \Delta_{k}(n) q^{n-\frac{k+1}{12}}.
\end{equation*}
For $\operatorname{Re}(z) > 0, j\in \mathbb{N}, h\in \mathbb{Z}$ with $\gcd(h, j) = 1$, and $h'\in\mathbb{Z}$ with
$hh' \equiv -1 \pmod j$, the Dedekind eta function satisfies the following transformation law \cite{Apostol-1990}:
\begin{equation}\label{eta-trans}
\eta\left(\frac{1}{j}(h+i z)\right)=e^{\frac{\pi i}{12 j}\left(h-h^{\prime}\right)} \omega_{h, j}^{-1} z^{-\frac{1}{2}} \eta\left(\frac{1}{j}\left(h^{\prime}+\frac{i}{z}\right)\right).
\end{equation}
To derive the transformation formula for $f_{k}(\tau)$ at the cusp $\frac{h}{j}$, we write
$$
e^{\frac{2 \pi i t}{j}(h+i z)}=e^{\frac{2 \pi i}{\operatorname{gcd}(j, t)}\left(h \frac{t}{\operatorname{gcd}(j, t)}+i \frac{t}{\operatorname{gcd}(j, t)} z\right)},
$$
where we note that $\gcd\left(h \frac{t}{\gcd(j, t)}, \frac{j}{\gcd(j, t)}\right)=1$. Therefore, using \eqref{eta-trans} with $j \mapsto \frac{j}{\gcd(j, t)}, h \mapsto$ $h \frac{t}{\gcd(j, t)}, z \mapsto \frac{t}{\gcd(j, t)} z$, we have
\begin{align*}
\eta\left(\frac{t}{j}(h+i z)\right)
=&e^{\frac{\pi i\gcd(j, t)}{12 j}\left(h\frac{t}{\gcd(j, t)} -h_{j,t}\right)}\omega_{h \frac{t}{\gcd(j, t)}, \frac{j}{\gcd(j, t)}}^{-1} \left(\frac{t}{\gcd(j, t)} z\right)^{-\frac{1}{2}}
\\
& \times\eta\left(\frac{\gcd(j, t)}{j}\left(h_{j, t}+i \frac{\gcd(j, t)}{t z}\right) \right).
\end{align*}
 It follows that
\begin{align*}
f_{k}\left(\frac{1}{j}(h+i z)\right)
=&\Omega_{k}(h,j)\cdot ze^{\frac{\pi i}{12j}\left((-2k-2)h-\gcd(j,2)h_{j, 2}-\gcd(j,2k+1)h_{j, 2k+1}+3h'+\gcd(j,4k+2)h_{j, 4k+2}\right)}
\\
&\times \frac{\eta\left(\frac{\gcd(j, 2)}{j}\left(h_{j, 2}+i \frac{\gcd(j, 2)}{2z}\right) \right)\eta\left(\frac{\gcd(j, 2k+1)}{j}\left(h_{j, 2k+1}+i \frac{\gcd(j, 2k+1)}{(2k+1) z}\right) \right)}{\eta\left(\frac{1}{j}\left(h^{\prime}+\frac{i}{z}\right)\right)^3
\eta\left(\frac{\gcd(j, 4k+2)}{j}\left(h_{j, 4k+2}+i \frac{\gcd(j, 4k+2)}{(4k+2) z}\right) \right)},
\end{align*}
where
$\Omega_{k}(h,j)$ is defined in \eqref{def:Omega}.
Applying the generating function of $k$-colored partitions
\begin{equation*}
\frac{1}{(q;q)_{\infty}^t}=\sum_{n=0}^{\infty}p_t(n)q^n
\end{equation*}
and the Euler's pentagonal number theorem \cite[(1.3.1)]{Andrews-1976}
\begin{equation*}
(q;q)_{\infty}=\sum_{n=-\infty}^{\infty}(-1)^nq^{n(3n-1)/2},
\end{equation*}
and along with \eqref{eta-def}, we have
\begin{align*}
f_{k}\left(\frac{1}{j}(h+i z)\right)
=&\Omega_{k}(h,j)\cdot ze^{-\frac{\pi i(k+1)h}{6j}}e^{\frac{\pi }{12jz}\cdot\frac{\left(\gcd(j,2)^2-2\right)\left(\gcd(j,2k+1)^2
-(2k+1)\right)+2(4k+2)}{4k+2}}
\\
&\times \sum_{n_1,n_2\geq 0}p_3(n_1)p(n_2)e^{\frac{2\pi i}{j}\left(n_1 h'+n_2\gcd(j,4k+2)h_{j,4k+2}\right)-\frac{2\pi }{jz}\left(n_1+\frac{\gcd(j,4k+2)^2}{4k+2}n_2\right)}
\\
&\times \sum_{n_3,n_4\in \mathbb{Z}}(-1)^{n_3+n_4}
e^{\frac{\pi i}{j}\left(n_3(3n_3-1)\gcd(j,2)h_{j,2}
+n_4(3n_4-1)\gcd(j,2k+1)h_{j,2k+1}\right)}
\\
&\qquad\qquad \times e^{-\frac{2\pi }{jz}\left(\frac{\gcd(j,2)^2}{4}n_3(3n_3-1)
+\frac{\gcd(j,2k+1)^2}{4k+2}n_4(3n_4-1)\right)}.
\end{align*}
The principal part of this expression is
\begin{align*}
&\Omega_{k}(h,j)\cdot ze^{-\frac{\pi i(k+1)h}{6j}}\sum_{\substack{n_1,n_2\geq 0 \\ n_3,n_4\in \mathbb{Z} \\r_{h,j}(n_1,n_2,n_3,n_4)\geq 0}}p_3(n_1)p(n_2)(-1)^{n_3+n_4}
e^{\frac{\pi i}{j}b_{h,j}(n_1,n_2,n_3,n_4)}e^{\frac{\pi }{12jz}r_{h,j}(n_1,n_2,n_3,n_4)},
\end{align*}
where
$b_{h,j}(n_1,n_2,n_3,n_4)$ and $r_{h,j}(n_1,n_2,n_3,n_4)$ are given in \eqref{def:b} and \eqref{def:r}, respectively.

Then in light of Theorem \ref{zuc-asym}, for $n-\frac{k+1}{12}>0$, we have
\begin{align*}
\Delta_{k}(n) &= \frac{2\pi}{n-\frac{k+ 1}{12}} \sum_{j=1}^{\infty} \frac{1}{j} \sum_{\substack{0 \leq h < j \\ \gcd(h,j) = 1}}\Omega_{k}(h,j) e^{-\frac{2\pi i h}{j}\left(n-\frac{k+ 1}{12}\right)}
\sum_{\substack{n_1,n_2\geq 0 \\ n_3,n_4\in \mathbb{Z} \\r_{h,j}(n_1,n_2,n_3,n_4)\geq 0}}p_3(n_1)p(n_2)(-1)^{n_3+n_4}
\\
&\quad \times e^{\frac{\pi i}{j}b_{h,j}(n_1,n_2,n_3,n_4)}
\frac{r_{h,j}(n_1,n_2,n_3,n_4)}{24} I_{2} \left( \frac{\pi}{j} \sqrt{\frac{2(n-\frac{k+ 1}{12})r_{h,j}(n_1,n_2,n_3,n_4)}{3}} \right)
\\
&=\frac{ \pi^3}{18 x^2_{k}(n)}\sum_{j=1}^{\infty} \frac{1}{j}
\sum_{\substack{n_1,n_2\geq 0 \\ n_3,n_4\in \mathbb{Z} \\r_{h,j}(n_1,n_2,n_3,n_4)\geq 0}}A_{j}(n,n_1,n_2,n_3,n_4)p_3(n_1)p(n_2)(-1)^{n_3+n_4}
\\
&\quad \times r_{h,j}(n_1,n_2,n_3,n_4) I_{2} \left( \frac{x_{k}(n)\sqrt{r_{h,j}(n_1,n_2,n_3,n_4)}}{j}  \right),
\end{align*}
where $A_{j}(n,n_1,n_2,n_3,n_4)$ and $x_{k}(n)$ are defined in \eqref{def:Aj} and \eqref{defi-xk}, respectively.
This completes the proof of Theorem \ref{exact}. \qed

According to Theorem \ref{exact}, we can prove Theorem \ref{asympform} by determining the main term and the errors.

\noindent{\it Proof of Theorem \ref{asympform}.}
Since $x^sI_s(x)$ is monotonically increasing as $x\rightarrow\infty$ for any fixed $s$, the main contribution comes from the term with largest value of $\frac{\sqrt{r_{h,j}(n_1,n_2,n_3,n_4)}}{j}$. It is clear that we need $n_1=n_2=n_3=n_4=0$. Then we should consider the largest value of
\begin{equation*}
\frac{\sqrt{r_{h,j}(0,0,0,0)}}{j}=\sqrt{\frac{\left(\gcd(j,2)^2-2\right)\left(\gcd(j,2k+1)^2
-(2k+1)\right)+2(4k+2)}{j^2(4k+2)}}.
\end{equation*}
If $j$ is even, then
\begin{equation*}
\frac{r_{h,j}(0,0,0,0)}{j^2}
=\frac{1}{j^2}\left(1+\frac{\gcd(j,2k+1)^2}{2k+1}\right),
\end{equation*}
which attains its maximum value $\frac{k+1}{4k+2}$ at $j=2$.
If $j$ is odd, then
\begin{equation*}
\frac{r_{h,j}(0,0,0,0)}{j^2}
=\frac{1}{j^2}\left(\frac{5}{2}-\frac{\gcd(j,2k+1)^2}{4k+2}\right),
\end{equation*}
which attains its maximum value $\alpha_k=\frac{5k+2}{2k+1}$ at $j=1$. The main term is thus obtained by taking $n_1=n_2=n_3=n_4=0$ and $j=1$, that is
\begin{equation*}
\frac{ \pi^3\alpha_k}{18 {x^2_{k}(n)}} I_2\left(\sqrt{\alpha_k } {x_{k}(n)}\right),
\end{equation*}
where $x_{k}(n)$ is defined in \eqref{defi-xk}. The remaining terms, including the terms arising from $j\geq 2$ and any $n_1,n_2,n_3,n_4$ and the terms coming from $j=1$ and $(n_1,n_2,n_3,n_4)\neq(0,0,0,0)$, are naturally the errors, i.e.,
\begin{align*}
R_{k}(n)=&\frac{ \pi^3}{18 x^2_{k}(n)}\sum_{j\geq 2 } \frac{1}{j}
\sum_{\substack{n_1,n_2\geq 0 \\ n_3,n_4\in \mathbb{Z} \\r_{h,j}(n_1,n_2,n_3,n_4)\geq 0}}A_{j}(n,n_1,n_2,n_3,n_4)p_3(n_1)p(n_2)(-1)^{n_3+n_4}
\\
&\quad \times r_{h,j}(n_1,n_2,n_3,n_4) I_{2} \left( \frac{x_{k}(n)\sqrt{r_{h,j}(n_1,n_2,n_3,n_4)}}{j}  \right)
\\
&+\frac{ \pi^3}{18 x^2_{k}(n)}
\sum_{\substack{n_1,n_2\geq 0, n_3,n_4\in \mathbb{Z} \\r_{h,1}(n_1,n_2,n_3,n_4)\geq 0\\ (n_1,n_2,n_3,n_4)\neq(0,0,0,0)}}A_{1}(n,n_1,n_2,n_3,n_4)p_3(n_1)p(n_2)(-1)^{n_3+n_4}
\\
&\quad \times r_{h,1}(n_1,n_2,n_3,n_4) I_{2} \left( x_{k}(n)\sqrt{r_{h,1}(n_1,n_2,n_3,n_4)}  \right),
\end{align*}
where we denote the associated sums by $R_{k}^{[1]}(n)$ and $R_{k}^{[2]}(n)$.
We next establish the upper bound for $|R_{k}(n)|$ by estimating $R_{k}^{[1]}(n)$ and $R_{k}^{[2]}(n)$ respectively.

We first consider $R_{k}^{[1]}(n)$. To this end, we need an estimate for $r_{h,j}(n_1,n_2,n_3,n_4)$.
If $2|j$, then $r_{h,j}(n_1,n_2,n_3,n_4)$ equals
\begin{equation*}
1+\frac{\gcd(j,2k+1)^2}{2k+1}-24\left(n_1
+\frac{2\gcd(j/2,2k+1)^2}{2k+1}n_2+
n_3(3n_3-1)+\frac{\gcd(j,2k+1)^2}{4k+2}n_4(3n_4-1)\right),
\end{equation*}
which can be nonnegative when $(n_1,n_3)\neq (0,0)$ and $n_2=n_4=0$ or $n_1=n_3=0$, $n_2\geq 0$ and $n_4\in \mathbb{Z}$. We also have
\begin{equation*}
r_{h,j}(n_1,n_2,n_3,n_4)\leq1+\frac{\gcd(j,2k+1)^2}{2k+1}\leq 2k+2.
\end{equation*}
If $2\nmid j$, then $r_{h,j}(n_1,n_2,n_3,n_4)$ equals
\begin{align*}
\frac{5}{2}-\frac{\gcd(j,2k+1)^2}{4k+2}
-24\left(n_1+\frac{\gcd(j,4k+2)^2}{4k+2}n_2+
\frac{n_3(3n_3-1)}4+\frac{\gcd(j,2k+1)^2}{4k+2}n_4(3n_4-1)\right),
\end{align*}
which can be nonnegative provided that $n_1=n_3=0$. We also have
\begin{equation*}
r_{h,j}(n_1,n_2,n_3,n_4)\leq \frac{5}{2}-\frac{\gcd(j,2k+1)^2}{4k+2} \leq\alpha_k.
\end{equation*}
Then combining with
\begin{equation}\label{A-bound}
 |A_{j}(n,n_1,n_2,n_3,n_4)|\leq j
\end{equation}
for any $n,n_1,n_2\geq 0$, $n_3,n_4\in\mathbb{Z}$ and $j\geq 1$,
we conclude that
\begin{align*}
|R_{k}^{[1]}(n)|
\leq &\frac{\pi^3}{18 {x^2_{k}(n)}}\left(
\left(2k+2\right)\sum_{\substack{j\geq 2\\2|j}}
I_{2} \left( \frac{x_{k}(n)}{j}\sqrt{1+\frac{\gcd(j,2k+1)^2}{2k+1}} \right)
\right.
\\
&\times\Bigg(\sum_{\substack{(n_1,n_3)\neq(0,0) \\n_1+n_3(3n_3-1)\leq \frac{k+1}{12}}}p_3(n_1)+\sum_{\substack{n_2\geq 0,n_4\in \mathbb{Z} \\4n_2+n_4(3n_4-1)\leq \frac{k+1}{6}}}p(n_2)\Bigg)
\\
&\left.+\alpha_k\sum_{\substack{j\geq 2\\2\nmid j }}I_{2} \left( \frac{x_{k}(n)}{j} \sqrt{\alpha_k} \right)
\sum_{\substack{n_2\geq 0,n_4\in \mathbb{Z} \\n_2+n_4(3n_4-1)\leq \frac{5k+2}{12}}}p(n_2)
\right).
\end{align*}
Using Lemma 2.2 of Bringmann, Kane, Rolen and Tripp \cite{Bringmann-Kane-Rolen-Tripp-2021}, we find that for $\kappa\in \mathbb{R}$ with $\kappa>-\frac{1}{2}$,
\begin{equation}\label{I-bound}
I_{\kappa}(s) \leq
\begin{cases}
\sqrt{\frac{2}{\pi s}} e^s, & \text{if}~s\geq 1,\\
\frac{2^{1-\kappa}s^\kappa}{\Gamma(\kappa+1)},& \text{if}~ 0\leq s<1.
\end{cases}
\end{equation}
Thus we obtain that
\begin{align}\label{I-bound-1}
&\sum_{\substack{j\geq 2\\2|j}}
I_{2} \left( \frac{x_{k}(n)}{j}\sqrt{1+\frac{\gcd(j,2k+1)^2}{2k+1}} \right)
=\sum_{j\geq 1}
I_{2} \left( \frac{x_{k}(n)}{2j}\sqrt{1+\frac{\gcd(j,2k+1)^2}{2k+1}} \right)
\notag\\
&\leq \sum_{j=1}^{\lfloor\frac{x_{k}(n)}{2}\sqrt{2k+2}\rfloor}
I_{2} \left( \frac{x_{k}(n)}{2}\sqrt{\frac{2k+2}{2k+1}} \right)+\sum_{j= \lfloor\frac{x_{k}(n)}{2}\sqrt{2k+2}\rfloor+1}^\infty I_{2} \left( \frac{x_{k}(n)\sqrt{2k+2}}{2j} \right)
\notag\\
&\leq(2k+2)^{\frac14}(2k+1)^{\frac14}\sqrt{\frac{x_k(n)}{\pi}}
\exp \left( \frac{x_{k}(n)}{2}\sqrt{\frac{2k+2}{2k+1}} \right)+\frac{(k+1)x_k^2(n)}{4\Gamma(3)}\sum_{j= \lfloor\frac{x_{k}(n)}{2}\sqrt{2k+2}\rfloor+1}^\infty\frac{1}{j^2}
\notag\\
&\leq\sqrt{\frac{(2k+2)x_k(n)}{\pi}}
\exp \left( \frac{x_{k}(n)}{2}\sqrt{\frac{2k+2}{2k+1}} \right)+\frac{\sqrt{k+1}}{4\sqrt{2}}x_k(n)
\notag\\
&\leq\sqrt{(k+1)x_k(n)}\exp \left( \frac{x_{k}(n)}{2}\sqrt{\frac{2k+2}{2k+1}} \right),
\end{align}
where in the last step, we use that $\sqrt{x}<e^{\frac x2}$ for $x\geq0$.
Similarly, since $\alpha_k=\frac{5k+2}{2k+1}<\frac 52$ and $\sqrt{x}<\frac{3}{4}e^{\frac x3}$ for $x\geq0$, we get
\begin{align}\label{I-bound-2}
\sum_{\substack{j\geq 2\\2\nmid j }}I_{2} \left( \frac{x_{k}(n)}{j} \sqrt{\alpha_k} \right)
&\leq\sum_{j=3}^{\lfloor x_{k}(n)\sqrt{\alpha_k}\rfloor}I_{2} \left( \frac{x_{k}(n)}{3} \sqrt{\alpha_k} \right)+\sum_{j= \lfloor x_{k}(n)\sqrt{\alpha_k}\rfloor+1}^\infty I_{2} \left( \frac{x_{k}(n)}{j} \sqrt{\alpha_k} \right)
\notag\\
&\leq \sqrt{\frac{6 x_{k}(n)}{\pi}} \alpha_k^{\frac14}
\exp \left( \frac{x_{k}(n)}{3} \sqrt{\alpha_k} \right)+\sqrt{\alpha_k}\cdot\frac{x_k(n)}{4}
\notag\\
&\leq 2\sqrt{x_{k}(n)}\exp \left( \frac{x_{k}(n)}{3} \sqrt{\alpha_k} \right).
\end{align}
Combining \eqref{I-bound-1} and \eqref{I-bound-2}, by determining the number of solutions for $n_1,n_2,n_3$ and $n_4$, for $k\geq1$, we have
\begin{align}\label{R1-bound}
|R_{k}^{[1]}(n)|
\leq& \frac{\pi^3}{18 {x^2_{k}(n)}}\left(\left(2k+2\right)
\sqrt{(k+1)x_k(n)}
\exp \left( \frac{x_{k}(n)}{2}\sqrt{\frac{2k+2}{2k+1}} \right)
\right.
\notag\\
&\times\Bigg(\sum_{n_1=0}^{\lfloor\frac{k+1}{12}\rfloor}
\left\lceil\frac{\sqrt{k+2}}{3}\right\rceil p_3(n_1)+\sum_{n_2=0}^{\lfloor\frac{k+1}{24}\rfloor}
\left\lceil\frac{\sqrt{2k+2}}{3}\right\rceil p(n_2)\Bigg)
\notag\\
&\left.+ 2\alpha_k\sqrt{x_{k}(n)}
\exp \left( \frac{x_{k}(n)}{3} \sqrt{\alpha_k} \right)
\sum_{n_2=0}^{\lfloor\frac{5k+2}{12}\rfloor}
\left\lceil\frac{\sqrt{5k +3}}{3}\right\rceil p(n_2)
\right).
\end{align}

We next turn our attention to $R_{k}^{[2]}(n)$. It's easy to see that
\begin{equation*}
r_{h,1}(n_1,n_2,n_3,n_4)=\frac{5k+2}{2k+1}-24\left(n_1+
\frac{n_2}{4k+2}+
\frac{n_3(3n_3-1)}4+\frac{n_4(3n_4-1)}{4k+2}\right),
\end{equation*}
which can be nonnegative when $n_1=n_3=0$. Since $n_1,n_2,n_3$ and
$n_4$ cannot all be $0$ simultaneously, $r_{h,1}(n_1,n_2,n_3,n_4)$ attains its maximum value $\beta_k=\frac{5k-10}{2k+1}$ at $n_1=n_3=n_4=0$ and $n_2=1$. It means that $r_{h,1}(n_1,n_2,n_3,n_4)\geq 0$ is no solution if $k\leq 2$. Based on \eqref{A-bound} and \eqref{I-bound}, for $k\geq 3$, we derive that
\begin{align}\label{R2-bound}
|R_{k}^{[2]}(n)|
&\leq \frac{\pi^3\beta_k}{18x^2_{k}(n)}
I_{2} \left( \sqrt{\beta_k}x_{k}(n) \right)
\sum_{\substack{n_2\geq 0,n_4\in \mathbb{Z} \\n_2+n_4(3n_4-1)\leq \frac{5k+2}{12}}}p(n_2)
\notag\\
&\leq \frac{\pi^3\beta_k}{18x^2_{k}(n)}
\left\lceil\frac{\sqrt{5k +3}}{3}\right\rceil
\sum_{m=0}^{\lfloor\frac{5k+2}{12}\rfloor}p(m)
I_{2} \left( \sqrt{\beta_k}x_{k}(n) \right)
\notag\\
&\leq \frac{\sqrt{2}\pi^{\frac{5}{2}}}{18x_{k}^{\frac{5}{2}}(n)} \beta_k^{\frac{3}{4}}
\left\lceil\frac{\sqrt{5k +3}}{3}\right\rceil
\sum_{m=0}^{\lfloor\frac{5k+2}{12}\rfloor}p(m)
\exp\left( \sqrt{\beta_k}x_{k}(n) \right).
\end{align}
It is evident that for $k \geq 3$,
\begin{equation*}
5k-10>\frac{5k+2}{9}\quad \text{and} \quad 5k-10>\frac{2k+2}{4},
\end{equation*}
and for $n>\frac{k+1}{12}$,
\begin{equation*}
x_{k}(n)\geq \frac{\pi}{6}.
\end{equation*}
Combining these with \eqref{R1-bound} and \eqref{R2-bound}, for $k\geq 3$ and $n\geq 1$, we have
\begin{align}\label{bound:Rk1}
&|R_{k}(n)|
\leq x_{k}^{-\frac{3}{2}}(n)
\exp\left( \sqrt{\beta_k} x_{k}(n)\right)
\Bigg[\frac{\pi^3}{9}\left(k+1\right)^{\frac32}
\left\lceil\frac{\sqrt{k +2}}{3}\right\rceil
\sum_{m=0}^{\lfloor\frac{k+1}{12}\rfloor}p_3(m)
+\sum_{m=0}^{\lfloor\frac{5k+2}{12}\rfloor}p(m)
\notag\\
&\qquad\quad~~\times\left(\frac{\pi^3}{9}\left(k+1\right)^{\frac32}
\left\lceil\frac{\sqrt{2k +2}}{3}\right\rceil+\left(2\alpha_k\cdot\frac{\pi^3}{18}
+\frac{\sqrt{2}\pi^{\frac{3}{2}}}{3} \beta_k^{\frac{3}{4}}\right)\left\lceil\frac{\sqrt{5k +3}}{3}\right\rceil\right)\Bigg].
\end{align}

Since $p_3(m)$ and $p(m)$ are nondecreasing in $m$, we have
\begin{equation}\label{eq:sums-comp}
\sum_{m=0}^{\lfloor\frac{k+1}{12}\rfloor}p_3(m)\le \left(\left\lfloor\tfrac{k+1}{12}\right\rfloor+1\right)
p_3(\left\lfloor\tfrac{k+1}{12}\right\rfloor),\qquad
\sum_{m=0}^{\lfloor\frac{5k+2}{12}\rfloor}p(m)\le \left(\left\lfloor\tfrac{5k+2}{12}\right\rfloor+1\right)
p(\left\lfloor\tfrac{5k+2}{12}\right\rfloor).
\end{equation}
Bringmann, Kane, Rolen and Tripp \cite[P. 626]{Bringmann-Kane-Rolen-Tripp-2021} gave a bound for $r$-colored partition function $p_r(m)$, that is, for all $m\ge1$,
\begin{equation*}
p_r(m)\ \le \exp\Big(\pi\sqrt{\tfrac{2rm}{3}}\Big).
\end{equation*}
Therefore, for all $m\ge1$,
\begin{equation}\label{eq:p-exp}
p(m)\ \le \exp\left(\pi\sqrt{\tfrac{2m}{3}}\right),\qquad
p_3(m)\ \le \exp\left(\pi\sqrt{2m}\right).
\end{equation}
For $k\ge3$,
\begin{align*}
\frac{k+1}{12}\leq \frac{k}{9},\quad \frac{k+1}{12}+1\leq \frac{k}{2},
\quad \frac{5k+2}{12}\leq \frac{k}{2},\quad \frac{5k+2}{12}+1\leq \frac{5k}{6},
\end{align*}
and together with \eqref{eq:sums-comp} and \eqref{eq:p-exp}, we obtain
\begin{equation}\label{eq:alg1}
\sum_{m=0}^{\lfloor\frac{k+1}{12}\rfloor}p_3(m) \le \frac{k}2\, e^{\pi\sqrt{2k/9}},\qquad
\sum_{m=0}^{\lfloor\frac{5k+2}{12}\rfloor}p(m)
\le \frac{5k}{6}\, e^{\pi\sqrt{k/3}}.
\end{equation}
Moreover, for $k\ge3$ we have the elementary bounds
\begin{align}
&\left(k+1\right)^{\frac32}
\left\lceil\frac{\sqrt{k +2}}{3}\right\rceil<\left(k+1\right)^{\frac32}
\left\lceil\frac{\sqrt{2k +2}}{3}\right\rceil\leq2k^2,\qquad
\left\lceil\frac{\sqrt{5k+3}}{3}\right\rceil\ \le\ \frac{3}{2}\sqrt{k}.
\end{align}
Finally, since $\alpha_k=\frac{5k+2}{2k+1}\in\big[\frac{17}{7},\,\frac52\big)$ and
$\beta_k=\frac{5k-10}{2k+1}\in\big[\frac{5}{7},\,\frac52\big)$ for $k\ge3$,
\begin{equation}\label{eq:alphabeta-bdd}
2\alpha_k\cdot\frac{\pi^3}{18}
+\frac{\sqrt{2}\pi^{\frac{3}{2}}}{3} \beta_k^{\frac{3}{4}}< 14.
\end{equation}

Plugging \eqref{eq:alg1}-\eqref{eq:alphabeta-bdd} into \eqref{bound:Rk1} yields
\begin{align*}
|R_{k}(n)|
&\leq \left(\frac{\pi^3k^3}{9}e^{\pi\sqrt{2k/9}}
+\left(\frac{2\pi^3k^2}{9}+\frac{42\sqrt{k}}{2}\right)\frac{5k}{6} e^{\pi\sqrt{k/3}}\right)x_{k}^{-\frac{3}{2}}(n)
\exp\left(\sqrt{\beta_k} x_{k}(n) \right)
\\
&\leq 12k^3 e^{\pi\sqrt{k/3}}x_{k}^{-\frac{3}{2}}(n)
\exp\left( \sqrt{\beta_k}x_{k}(n) \right)
\end{align*}
for $k\geq 3$. This completes the proof.
\qed

\section{Inequalities of any order}\label{sec:ineq-any}
In this section, we shall prove the Tur\'an and Laguerre inequalities of any order and asymptotically complete monotonicity by utilizing the asymptotic formula of $\Delta_{k}(n)$ stated in Theorem \ref{asympform}.

Griffin, Ono, Rolen and Zagier \cite{Griffin-Ono-Rolen-Zagier-2019} gave the following criteria for the positivity of the Tur\'an inequalities of any order for a sequence $\{\alpha(n)\}$.

\begin{thm}[Griffin, Ono, Rolen and Zagier]\label{gorz}
Let $\{\alpha(n)\}$, $\{A(n)\}$, and $\{\delta(n)\}$ be sequences of positive real numbers with $\delta(n)$ tending to zero and satisfying
\[
\log \left(\frac{\alpha(n+j)}{\alpha(n)}\right)=A(n)j-\delta^2(n)j^2
+o\left(\delta^d(n)\right) \quad \mbox{as}\quad n\rightarrow \infty
\]
for some $d\geq1$ and all $0\leq j\leq d$. Then, we have
\[
\lim_{n\rightarrow \infty}\left(\frac{\delta^{-d}(n)}{\alpha(n)}J_{\alpha}^{d,n}\left(\frac{\delta(n)X-1}{\exp
(A(n))}\right)\right)=H_d(X)
\]
uniformly for $X$ in any compact subset of $\mathbb{R}$, where $H_d(X)$ are the Hermite polynomials. Furthermore, this
implies that the polynomials $J_{\alpha}^{d,n}(X)$ are hyperbolic for all but finitely many values of $n$.
\end{thm}
For the Laguerre inequalities of any order, Wang and Yang \cite{Wang-Yang-2024} gave the following result.

\begin{thm}[Wang and Yang]\label{theowy4}
Let $\{\alpha(n)\},\{\delta(n)\},\{A_m(n)\}$ be sequences such that as $n\rightarrow\infty$, $\alpha(n)$ stays positive, $\delta(n)\rightarrow 0^+$, $A_2(n)<0$, $A_2(n)=\Theta(\delta^{t}(n))$ and $A_{2m}(n)=o(\delta^{mt}(n))\ (m\geq2)$ for some positive $t$, and for $-r<j<r$,
\begin{equation*}
\log\left(\frac{\alpha(n+j)}{\alpha(n)}\right)
=\sum_{m=1}^{2N}A_m(n)j^m+o\left(\delta^{tr}(n)\right),
\end{equation*}
then for sufficiently large $n$,
\begin{equation*}
L_r(\alpha(n))=\frac{1}{2}\sum_{j=0}^{2r}(-1)^{j+r}{2r\choose j}\alpha(n+j)\alpha(n+2r-j)>0.
\end{equation*}
\end{thm}
They also gave the criteria for the
asymptotically complete monotonicity of the sequence $\{\alpha(n)\}$.
\begin{thm}[Wang and Yang]\label{thm:wy-d}
Let $\{\alpha(n)\},\{\delta(n)\},\{A_m(n)\}$ be sequences such that as $n\rightarrow\infty$, $\alpha(n)$ stays positive, $\delta(n)\rightarrow 0^+$, and
\begin{equation*}
\log\left(\frac{\alpha(n+j)}{\alpha(n)}\right)
=\sum_{m=1}^{N}A_m(n)j^m
+R(n,j),
\end{equation*}
where $A_1(n)=\Theta(\delta^{t}(n))$, $A_{m}(n)=o(\delta^{mt}(n))\ (m\geq2)$ and $R(n,j)=o\left(\delta^{rt}(n)\right)$ for some positive $r$ and $t$, then for sufficiently large $n$,
\begin{equation*}
A_1(n)^r\Delta^r\left(\alpha(n)\right)>0.
\end{equation*}
\end{thm}

%

In light of Theorems \ref{gorz}-\ref{thm:wy-d}, to establish all the desired inequalities it suffices to obtain an asymptotic expansion for $\log\left(\frac{\Delta_k(n+j)}{\Delta_k(n)}\right)$.
According to Theorem \ref{asympform} and the result of Dong, Ji, and Jia \cite[Theorem 1.2]{Dong-Ji-Jia-2023}, we have, as $n\rightarrow \infty$,
\begin{equation*}
\Delta_k(n)=\frac{ \pi^3\alpha_k}{18 {x^2_{k}(n)}} I_2\left(\sqrt{\alpha_k } {x_{k}(n)}\right)+
\begin{cases}
O\left(x_{k}^{-\frac{3}{2}}(n)
\exp\left(\sqrt{\beta_k}x_{k}(n) \right)\right),& \text{if}~k\geq3,
\\[6pt]
O\left(x_{k}^{-\frac{5}{2}}(n)
\exp\left(\frac{\sqrt{\alpha_k}}2x_{k}(n) \right)\right),& \text{if }~k=1,2.
\end{cases}
\end{equation*}
From the DLMF \cite[(10.30.4)]{Olver-Lozier-Boisvert-Clark-2010}, we know that when $\nu$ is fixed and $z\rightarrow \infty$,
\begin{equation*}
  I_{\nu}(z)\sim \frac{e^z}{\sqrt{2\pi z}},
\end{equation*}
which gives that for $k\geq 1$, as $n\rightarrow \infty$,
\begin{equation*}
\Delta_k(n)\sim\frac{ \pi^{\frac{5}{2}}\alpha_k^{\frac{3}{4}}}{18\sqrt{2} {x^{\frac{5}{2}}_{k}(n)}} \exp\left(\sqrt{\alpha_k } {x_{k}(n)}\right).
\end{equation*}
Representing $x_{k}(n+j)$ by $x_{k}(n)$ as
\begin{equation*}
 x_{k}(n+j)=\sqrt{x_{k}^2(n)+\frac{2\pi^2}{3}j},
\end{equation*}
and applying Taylor's theorem, we have
\begin{align}\label{asym:log-delta}
&\log\left(\frac{\Delta_k(n+j)}{\Delta_k(n)}\right)
\notag\\
&\sim \sqrt{\alpha_k }\left(\sqrt{x_{k}^2(n)+\frac{2\pi^2}{3}j}
-x_{k}(n)\right)-\frac{5}{4}\log\left(1+\frac{2\pi^2}{3x_{k}^2(n)}j\right)
\notag\\
&=\sqrt{\alpha_k }\sum_{m=1}^\infty{\frac{1}{2} \choose m}\left(\frac{2\pi^2}{3}\right)^m\frac{j^m}{x_{k}^{2m-1}(n)}
-\frac{5}{4}\sum_{m=1}^\infty\frac{(-1)^{m+1}}{m}
\left(\frac{2\pi^2}{3x_{k}^2(n)}\right)^mj^m
\notag\\
&=\left(\frac{\sqrt{\alpha_k }\pi^2}{3x_{k}(n)}-\frac{5\pi^2}{6x_{k}^2(n)}\right)j+
\sum_{m=2}^\infty\left(-\frac{\sqrt{\alpha_k }(2m-3)!!}{2^m m!x_{k}^{2m-1}(n)}+\frac{5}{4mx_{k}^{2m}(n)}\right)
\left(-\frac{2\pi^2}{3}\right)^m j^m.
\end{align}

We first verify Theorem \ref{thm:ineq-any-order} by employing \eqref{asym:log-delta} and Theorems \ref{gorz} and \ref{theowy4}.

{\noindent\it{Proof of Theorem \ref{thm:ineq-any-order}.}}
Note that if a sequence $\{\alpha(n)\}$ satisfies the conditions in Theorem \ref{gorz}, then it also satisfies the conditions in Theorem \ref{theowy4}. To confirm $\Delta_k(n)$ satisfies the Tur\'an and Laguerre inequalities of any order, we only need to show it satisfies all the conditions in Theorem \ref{gorz}. From \eqref{asym:log-delta}, we know that the sequence $\{\Delta_k(n)\}$ satisfies the hypotheses of
Theorem \ref{gorz} with
\[A(n)=\frac{\sqrt{\alpha_k }\pi^2}{3x_{k}(n)}+O\left(\frac{1}{x_{k}^2(n)}\right)
\quad \text{and} \quad
\delta(n)=\frac{\alpha_k^{\frac 14} \pi^2}{3\sqrt{2}x_{k}^{\frac32}(n)}+O\left(\frac{1}{x_{k}^{\frac52}(n)}\right).\]
This completes the proof.
\qed

We proceed by applying \eqref{asym:log-delta} again, together with Theorem \ref{thm:wy-d}, to confirm Theorem \ref{thm:ineq-any-order-d}.

{\noindent\it{Proof of Theorem \ref{thm:ineq-any-order-d}.}}
It suffices to prove that the sequence $\{\Delta_k(n)\}$ satisfies the hypotheses of Theorem \ref{thm:wy-d}.
Set
\begin{equation*}
\delta(n)=\frac{\sqrt{\alpha_k }\pi^2}{3x_{k}(n)}+O\left(\frac{1}{x_{k}^2(n)}\right).
\end{equation*}
Applying \eqref{asym:log-delta} gives that
\begin{align*}
\log\left(\frac{\Delta_k(n+j)}{\Delta_k(n)}\right)
=\sum_{m=1}^{r-1}A_m(n)j^m+R(n,j),
\end{align*}
where
\[A_1(n)=\frac{\sqrt{\alpha_k }\pi^2}{3x_{k}(n)}-\frac{5\pi^2}{6x_{k}^2(n)}=\Theta(\delta(n)),
\quad
R(n,j)=O(\delta(n)^{2r-1})=o(\delta(n)^r),\]
and for $m\geq 2$,
\[
A_m(n)=\left(-\frac{\sqrt{\alpha_k }(2m-3)!!}{2^m m!x_{k}^{2m-1}(n)}+\frac{5}{4mx_{k}^{2m}(n)}\right)
\left(-\frac{2\pi^2}{3}\right)^m
=\Theta(\delta(n)^{2m-1})=o(\delta(n)^m).\]
Based on Theorem \ref{thm:wy-d}, it follows that for sufficiently large $n$,
\begin{equation*}
  A_1(n)^r\Delta^r\left(\Delta_k(n)\right)>0,
\end{equation*}
which implies that $\Delta^r\left(\Delta_k(n)\right)>0$.
\qed

\section{Log-concavity}\label{sec:log}
In this section, we show that $\Delta_k(n)$ is log-concave in a certain range by applying Theorem \ref{asympform}. To this end, we need some lemmas. Let
\[
\Theta(n):=\frac{\Delta_k(n-1)\,\Delta_k(n+1)}{\Delta_k^2(n)},
\qquad
\Theta_M(n):=\frac{M_k(n-1)\,M_k(n+1)}{M_k^2(n)},
\qquad
\varepsilon_n:=\frac{|R_k(n)|}{M_k(n)}.
\]
We first give a sufficient condition for $\Theta(n)\le 1$, i.e., $\Delta_k^2(n) \ge \Delta_k(n-1)\,\Delta_k(n+1)$.

\begin{lem}\label{lem:factor}

Assume $\varepsilon_{n-1},\varepsilon_n,\varepsilon_{n+1}<1$. Then
\begin{equation}\label{eq:factor-consequence-exact}
\Theta(n)\le\Theta_M(n)\left(1+\frac{4\varepsilon_n^\ast}
{(1-\varepsilon_n^\ast)^2}\right),
\end{equation}
where $\varepsilon_n^\ast:=\max\{\varepsilon_{n-1},\varepsilon_n,\varepsilon_{n+1}\}$.
In particular, a sufficient condition for $\Theta(n)\le 1$ is
\begin{equation}\label{eq:gap-vs-error-exact}
1-\Theta_M(n)\ge\frac{4\varepsilon_n^\ast}{(1-\varepsilon_n^\ast)^2}.
\end{equation}
Moreover, if $\varepsilon_n^\ast\le \tfrac12$, it suffices that
\begin{equation}\label{eq:gap-vs-error-crude}
1-\Theta_M(n)\ge16\varepsilon_n^\ast.
\end{equation}
\end{lem}

\begin{proof}

By the decomposition $\Delta_k=M_k+R_k$ and the definition of $\varepsilon_n$,
\[
\Delta_k(n)\le M_k(n)\bigl(1+\varepsilon_n\bigr),\qquad
\Delta_k(n)\ge M_k(n)\bigl(1-\varepsilon_n\bigr),
\]
for every integer $n$. In particular,
\[
\Delta_k(n-1)\le M_k(n-1)(1+\varepsilon_{n-1}),\qquad
\Delta_k(n+1)\le M_k(n+1)(1+\varepsilon_{n+1}).
\]
Then
\[
\Theta(n)=\frac{\Delta_k(n-1)\Delta_k(n+1)}{\Delta_k^2(n)}
\le\frac{M_k(n-1)M_k(n+1)}{M_k^2(n)}\cdot
\frac{(1+\varepsilon_{n-1})(1+\varepsilon_{n+1})}{(1-\varepsilon_n)^2}.
\]

Since $0\le \varepsilon_{n\pm1},\varepsilon_n^\ast<1$, we have
\[
1+\varepsilon_{n\pm1}\le1+\varepsilon_n^\ast,\qquad
1-\varepsilon_n\ge1-\varepsilon_n^\ast>0.
\]
As $u\mapsto u^{-2}$ is decreasing on $(0,\infty)$, we obtain
\[
\frac{(1+\varepsilon_{n-1})(1+\varepsilon_{n+1})}{(1-\varepsilon_n)^2}
\le
\frac{(1+\varepsilon_n^\ast)^2}{(1-\varepsilon_n^\ast)^2}.
\]

For any $\varepsilon\in[0,1)$,
\[
\frac{(1+\varepsilon)^2}{(1-\varepsilon)^2}
= 1+\frac{4\varepsilon}{(1-\varepsilon)^2}.
\]
Applying this with $\varepsilon=\varepsilon_n^\ast$ gives \eqref{eq:factor-consequence-exact}.

From \eqref{eq:factor-consequence-exact},
\[
\Theta(n)\le\Theta_M(n)\bigl(1+A\bigr),\qquad
A:=\frac{4\varepsilon_n^\ast}{(1-\varepsilon_n^\ast)^2}.
\]
If $A\le 1-\Theta_M(n)$, then
\[
\Theta_M(n)(1+A)\le\Theta_M(n)+\Theta_M(n)\bigl(1-\Theta_M(n)\bigr)
=2\Theta_M(n)-\Theta_M(n)^2\le1,
\]
because $t(2-t)\le1$ for all $t\ge0$ (equivalently, $(t-1)^2\ge0$).
This proves \eqref{eq:gap-vs-error-exact}. Finally, if $\varepsilon_n^\ast\le\tfrac12$, then
\[
A=\frac{4\varepsilon_n^\ast}{(1-\varepsilon_n^\ast)^2}\le\frac{4\varepsilon_n^\ast}{(1/2)^2}=16\varepsilon_n^\ast,
\]
so the cruder sufficient condition \eqref{eq:gap-vs-error-crude} follows.
\end{proof}

According to Lemma \ref{lem:factor}, we need an upper bound for $\varepsilon_n$ in order to prove the log-concavity of $\Delta_k(n)$.

\begin{lem}\label{lem:relerr}
Let $\delta_k:=\sqrt{\alpha_k}-\sqrt{\beta_k}>0$. For $\sqrt{\alpha_k}\,x_k(n)\ge4$, we have
\begin{equation}\label{eq:eps}
\varepsilon_n \le C_k\,x_k(n)\,e^{-\delta_k\,x_k(n)},\qquad
C_k:=\frac{432\sqrt{2}k^3 e^{\pi\sqrt{k/3}}}{\pi^{\frac{5}{2}}\alpha_k^{\frac{3}{4}}}.
\end{equation}
\end{lem}

\begin{proof}
From \eqref{eq:M-R}, it follows that
\begin{equation}\label{eq:eps-start}
\varepsilon_n=\frac{|R_k(n)|}{M_k(n)}
 \le \frac{216k^3 e^{\pi\sqrt{k/3}}}{\pi^3\alpha_k}
 \cdot\frac{x_{k}^{\frac{1}{2}}(n)
\exp\left( \sqrt{\beta_k}x_{k}(n) \right)}{ I_2\left( \sqrt{\alpha_k}x_{k}(n) \right)}.
\end{equation}

We claim that for all $t\ge4$,
\begin{equation}\label{eq:I2-lb}
I_2(t)\ \ge\ \frac{e^{t}}{2\sqrt{2\pi t}}.
\end{equation}
To see this, use the large-argument asymptotic expansion (fixed order $\nu$) from the DLMF \cite[\S10.40(i)-(ii)]{Olver-Lozier-Boisvert-Clark-2010}, with the remainder controlled in absolute value (see also \cite{Nemes-2017}). For $\nu=2$ and $t\ge4$,
\[
I_2(t)=\frac{e^{t}}{\sqrt{2\pi t}}
\left(1-\frac{15}{8t}+\frac{105}{128\,t^{2}}+R_{3}(t)\right),
\qquad |R_{3}(t)|\le \frac{945}{3072\,t^{3}}.
\]
Since the third coefficient is positive, to obtain a uniform lower bound we discard the positive term $+\frac{105}{128\,t^2}$ (worst case) and replace $+R_3(t)$ by $-|R_3(t)|$. Hence, for $t\ge4$,
\[
\frac{I_2(t)}{e^{t}/\sqrt{2\pi t}}
\ \ge\ 1-\frac{15}{8t}-\frac{945}{3072\,t^{3}}
\ \ge\ 1-\frac{15}{32}-\frac{945}{196{,}608}
=\frac{103{,}503}{196{,}608}\ >\ \frac12,
\]
which proves \eqref{eq:I2-lb}.

By \eqref{eq:eps-start} and \eqref{eq:I2-lb}, for $\sqrt{\alpha_k}\,x_k(n)\ge4$,
\[
\varepsilon_n \le \frac{432\sqrt{2}k^3 e^{\pi\sqrt{k/3}}}{\pi^{\frac{5}{2}}\alpha_k^{\frac{3}{4}}}
\,x_k(n)\exp\left(-\left(\sqrt{\alpha_k}-\sqrt{\beta_k}\right)x_k(n)\right),
\]
which is \eqref{eq:eps}.
\end{proof}

In order to prove the log-concavity of $\Delta_k(n)$, based on Lemma \ref{lem:factor}, we also need a lower bound for $1-\Theta_M(n)$.

\begin{lem}\label{lem:concavity}
For all $n$ with $x_k(n-1)\ge \dfrac{26}{\sqrt{\alpha_k}}$
, we have
\begin{equation}\label{eq:gap}
1-\Theta_M(n) \ge \frac{\sqrt{\alpha_k}\,\pi^4}{36\,x_k^3(n+1)}.
\end{equation}
\end{lem}

\begin{proof}
Let $f(n):=\log \left(M_k(n)\right)$, and regard $n$ as a continuous variable. Then
\[
\Delta^2 f(n):=f(n+1)-2f(n)+f(n-1)
=\int_{-1}^{1}(1-|s|)\,f''(n+s)\,ds,
\]
\emph{(by the integral form of the second-order Taylor remainder / Peano kernel; see, e.g., \cite{Davis-1975})}
hence
\begin{equation}\label{eq:delta2f-upper}
\Delta^2 f(n)\le\sup_{|s|\le 1} f''(n+s),\qquad
\Theta_M(n)=\exp\big(\Delta^2 f(n)\big).
\end{equation}
Write
\[
f(n)=\log\left(\frac{\alpha_k\pi^{3}}{18}\right)-2\log\left( x_k(n)\right)+\phi(t),
\]
with $\phi(t):=\log \left(I_2(t)\right)$ and $t:=\sqrt{\alpha_k}\,x_k(n)$. Using
\[
x_k'(n)=\frac{\pi^2}{3x_k(n)},\qquad x_k''(n)=-\frac{\pi^4}{9x_k(n)^3},\qquad
t'=\sqrt{\alpha_k}\,x_k'(n),\quad t''=\sqrt{\alpha_k}\,x_k''(n),
\]
a straightforward calculation yields
\begin{equation}\label{eq:fpp}
f''(n)=\frac{4\pi^4}{9x_k(n)^4}
+\phi''(t)\frac{\alpha_k\pi^4}{9x_k(n)^2}
-\phi'(t)\frac{\sqrt{\alpha_k}\,\pi^4}{9x_k(n)^3}.
\end{equation}

We claim that for $t\ge8$,
\begin{equation}\label{eq:phi-bds-safe}
\Big|\phi'(t)-\Big(1-\frac{1}{2t}\Big)\Big|\le \frac{2}{t^2},\qquad
\Big|\phi''(t)-\frac{1}{2t^2}\Big|\le \frac{4}{t^3}.
\end{equation}
By \cite[\S10.40(i)-(ii)]{Olver-Lozier-Boisvert-Clark-2010} together with the error estimates in \cite[\S 10.41]{Olver-Lozier-Boisvert-Clark-2010}, for $t\ge4$,
\[
I_\nu(t)=\frac{e^{t}}{\sqrt{2\pi t}}\!\left(1-\frac{\mu-1}{8t}
+\frac{(\mu-1)(\mu-9)}{2!\,8^2\,t^2}+R_{3}^{(\nu)}(t)\right),
\quad \mu:=4\nu^2,
\]
with a remainder $R_{3}^{(\nu)}(t)=O(t^{-3})$; for fixed $\nu$ and $t\ge4$ this implies
$|R_{3}^{(\nu)}(t)|\le C_\nu/t^3$ for some explicit constant $C_\nu$ (e.g. bounded by a constant multiple of the next coefficient).
Specializing to $\nu=2$ and $\nu=3$ gives
\begin{align*}
I_2(t)&=\frac{e^{t}}{\sqrt{2\pi t}}\left(1-\frac{15}{8t}+\frac{105}{128t^2}+R_3^{(2)}(t)\right),
&|R_3^{(2)}(t)|&\le \frac{C_2}{t^3},\\
I_3(t)&=\frac{e^{t}}{\sqrt{2\pi t}}\left(1-\frac{35}{8t}+\frac{945}{128t^2}+R_3^{(3)}(t)\right),
&|R_3^{(3)}(t)|&\le \frac{C_3}{t^3},
\end{align*}
for some explicit $C_2,C_3$ (one may take, for instance, the absolute value of the first neglected coefficient times a harmless factor; any concrete choice works for $t\ge4$).

Using the standard identity
\[
I_\nu'(t)=I_{\nu+1}(t)+\frac{\nu}{t}I_\nu(t)
\qquad\text{(see DLMF \cite[(10.29.2)]{Olver-Lozier-Boisvert-Clark-2010})},
\]
we obtain for $\nu=2$
\[
I_2'(t)=I_3(t)+\frac{2}{t}I_2(t).
\]
Substituting the expansions and collecting terms up to $t^{-2}$ gives
\[
I_2'(t)=\frac{e^{t}}{\sqrt{2\pi t}}\left(1-\frac{19}{8t}+\frac{465}{128t^2}+S_3(t)\right),
\qquad |S_3(t)|\le \frac{C'}{t^3},
\]
for some explicit $C'$ depending only on the fixed orders (here $2,3$).
Dividing by the corresponding expansion of $I_2(t)$ and using
\(
\dfrac{1+a}{1+b}=(1+a)\,(1-b+b^2-\cdots)
\)
with $a,b=O(t^{-1})$ (and the above remainder bounds), a short algebraic check yields
\[
\phi'(t)=\frac{I_2'(t)}{I_2(t)}
=1-\frac{1}{2t}+\frac{15}{8t^2}+E_3(t),\qquad |E_3(t)|\le \frac{1}{t^3}\quad(t\ge8).
\]
In particular, for $t\ge8$
\[
\left|\phi'(t)-\left(1-\frac{1}{2t}\right)\right|\le \frac{15}{8t^2}+|E_3(t)|
\le \frac{2}{t^2},
\]
proving the first inequality in \eqref{eq:phi-bds-safe}.

The modified Bessel equation
\[
t^2 I_\nu''+t I_\nu'-(t^2+\nu^2)I_\nu=0 \qquad\text{(see DLMF \cite[\S 10.25(i)]{Olver-Lozier-Boisvert-Clark-2010})}
\]
implies the Riccati identity for $\phi(t)=\log \left(I_2(t)\right)$:
\[
\phi''(t)=1+\frac{4}{t^2}-\frac{\phi'(t)}{t}-\big(\phi'(t)\big)^2.
\]
Write $\phi'(t)=1-\tfrac{1}{2t}+u(t)$, where by the first inequality in \eqref{eq:phi-bds-safe}, we have $|u(t)|\le 2/t^2$ and, more sharply, $u(t)=\tfrac{15}{8t^2}+O(t^{-3})$. A direct substitution gives
\[
\phi''(t)=\frac{1}{2t^2}-2\,v(t)-u(t)^2,\qquad
\text{where } v(t):=u(t)-\frac{15}{8t^2}.
\]
Hence $|v(t)|\le 1/t^3$ (for $t\ge8$) and $u(t)^2\le 4/t^4\le 1/t^3$, which yields
\[
\left|\phi''(t)-\frac{1}{2t^2}\right|\le 2|v(t)|+u(t)^2\le \frac{4}{t^3}
\]
for $t\ge8$. This proves \eqref{eq:phi-bds-safe}.

Using \eqref{eq:fpp} and the upper bounds from \eqref{eq:phi-bds-safe}, with $t=\sqrt{\alpha_k}\,x_k(n)$, we have
\begin{equation*}
f''(n)\ \le\ -\frac{\sqrt{\alpha_k}\pi^4}{9x_k^3(n)}\ +\ \frac{5\pi^4}{9x_k^4(n)}\ +\ \frac{2}{3}\cdot\frac{\pi^4}{\sqrt{\alpha_k}\,x_k^5(n)},
\end{equation*}
whenever $t\ge8$ \emph{(and in particular for $t\ge 26$)}. Applying the same bound to $f''(n+s)$ (with $x_k(n)$ replaced by $x_k(n+s)$ and $t$ by $\sqrt{\alpha_k}\,x_k(n+s)$), and then taking the supremum in \eqref{eq:delta2f-upper}, the monotonicity in $x_k(n)$ yields
\[
\mathcal{D}(n):=\Delta^2 f(n) \le -\frac{\sqrt{\alpha_k}\pi^4}{9\,x_k^3(n+1)}
\ +\ \frac{5\pi^4}{9\,x_k^4(n-1)}
\ +\ \frac{2}{3}\cdot\frac{\pi^4}{\sqrt{\alpha_k}\,x_k^5(n-1)}.
\]
As in the original argument, it suffices to ensure that the positive remainders are at most one half of the main (negative) term, namely
\[
\frac{5\pi^4}{9\,x_k^4(n-1)}+\frac{2}{3}\cdot\frac{\pi^4}{\sqrt{\alpha_k}
\,x_k^5(n-1)}\le \frac{1}{2}\cdot\frac{\sqrt{\alpha_k}\pi^4}{9\,x_k^3(n+1)},
\]
which holds for all $x_k(n-1)\ge 26/\sqrt{\alpha_k}$ (equivalently, $t(n-1)\ge 26$). Consequently
\[
\mathcal{D}(n)\le -\frac{\sqrt{\alpha_k}\pi^4}{18\,x_k^3(n+1)} < 0.
\]

From \eqref{eq:phi-bds-safe} we also have, for $t\ge8$,
\[
\phi''(t)\ \ge\ \frac{1}{2t^2}-\frac{4}{t^3},\qquad
\phi'(t)\ \le\ 1-\frac{1}{2t}+\frac{2}{t^2}.
\]
Plugging these into \eqref{eq:fpp} yields the matching lower bound
\[
f''(n)\ \ge\ -\frac{\sqrt{\alpha_k}\pi^4}{9x_k^3(n)}\ -\ \frac{5\pi^4}{9x_k^4(n)}\ -\ \frac{2}{3}\cdot\frac{\pi^4}{\sqrt{\alpha_k}\,x_k^5(n)},
\]
so that $|f''(n)|\le c/x_k^3(n)$ for some explicit constant $c$ when $t\ge 26$. Therefore,
\(
|\mathcal{D}(n)|
=\left|\int_{-1}^{1}(1-|s|)f''(n+s)\,ds\right|
\le c/x_k^3(n)<1
\)
for $x_k(n-1)\ge 26/\sqrt\alpha_k$.

Finally, for $|\mathcal{D}(n)|<1$,
\(
1-e^{\mathcal{D}(n)}\ge -\mathcal{D}(n)-\tfrac{\mathcal{D}^2(n)}{2}\ge \tfrac{-\mathcal{D}(n)}{2}.
\)
Hence
\[
1-\Theta_M(n)\ =\ 1-e^{\Delta^2 f(n)}\ =\ 1-e^{\mathcal{D}(n)}\ \ge\ \frac{1}{2}\cdot\frac{\sqrt{\alpha_k}\pi^4}{18\,x_k^3(n+1)}
=\frac{\sqrt{\alpha_k}\pi^4}{36\,x_k^3(n+1)},
\]
which is \eqref{eq:gap}.
\end{proof}

In light of Lemmas \ref{lem:factor}-\ref{lem:concavity}, we derive another sufficient condition for the log-concavity of $\Delta_k(n)$, which can be used directly to prove Theorem \ref{thm:main}.
\begin{lem}\label{lem:analytic}
Define
$A_k:=\frac{1728\,C_k}{\pi^4\sqrt{\alpha_k}}$.
If $x_k(n-1)\ge \dfrac{26}{\sqrt{\alpha_k}}$ and
\begin{equation}\label{eq:key-ineq}
e^{\delta_k\,x_k(n)}\ \ge\ A_k x_k^4(n)
\qquad\big(\text{equivalently }~ \delta_k x_k(n)\ge \log A_k+ 4\log x_k(n)\big),
\end{equation}
then $\Delta_k^2(n)\ge \Delta_k(n-1)\Delta_k(n+1)$.
\end{lem}

\begin{proof}
By Lemma \ref{lem:relerr}, for $\sqrt{\alpha_k}\,x_k(j)\ge 4$ we have
\begin{equation*}
\varepsilon_j\ =\ \frac{|R_k(j)|}{M_k(j)}\ \le\ C_k\,x_k(j)\,e^{-\delta_k\,x_k(j)}\qquad(j=n-1,n,n+1).
\end{equation*}
Since $x_k(n)$ is increasing in $n$, we have $x_k(n-1)\le x_k(n)\le x_k(n+1)$ and therefore
\begin{equation}\label{eq:eps-star-bound}
\varepsilon_n^\ast=\max\{\varepsilon_{n-1},\varepsilon_n,\varepsilon_{n+1}\}
\le C_k\, x_k(n+1)\, e^{-\delta_k x_k(n-1)}.
\end{equation}
Using $x_k(n)=\frac{\pi}{6}\sqrt{24n-(2k+2)}$, we have
\[
x_k(n\pm1)=x_k(n)\sqrt{1\pm\frac{24}{24n-(2k+2)}}=x_k(n)\sqrt{1\pm\frac{2\pi^2}
{3x_k^2(n)}},
\]
hence
\[
x_k(n)-x_k(n-1)=x_k(n)\Big(1-\sqrt{1-\tfrac{2\pi^2}{3x_k^2(n)}}\Big)\le \frac{2\pi^2}{3x_k(n)}.
\]
Therefore,
\[
e^{\delta_k(x_k(n)-x_k(n-1))}\ \le\ \exp\!\Big(\frac{2\delta_k\pi^2}{3x_k(n)}\Big)
\ \le\ \exp\!\Big(\frac{2\sqrt{\alpha_k}\,\pi^2}{3x_k(n)}\Big)
\ \le\ \exp\!\Big(\frac{\alpha_k\pi^2}{39}\Big)=:\rho_k
\]
and
\[x_k(n+1)\le x_k(n)\sqrt{1+\frac{2\pi^2}
{3}\cdot \frac{\alpha_k}{26^2}}< \frac{3}{2}x_k(n), \]
where we use that $\delta_k\le \sqrt{\alpha_k}$, $x_k(n)\ge 26/\sqrt{\alpha_k}$ and $\alpha_k<\frac{5}{2}$. These together with \eqref{eq:eps-star-bound} give that
\begin{equation}\label{eq:eps-star}
\varepsilon_n^\ast\le C_k\, x_k(n+1)\, e^{\delta_k (x_k(n)-x_k(n-1)-x_k(n))} < \frac{3}{2}\,\rho_k\,C_k\, x_k(n)\, e^{-\delta_k x_k(n)}.
\end{equation}

Assume \eqref{eq:key-ineq}. Then by \eqref{eq:eps-star} and $x_k(n)\ge 26/\sqrt{\alpha_k}$,
\[
\varepsilon_n^\ast < \frac{3\rho_k\,C_k}{2A_k\,x_k^3(n)}
=\frac{\pi^4\rho_k\sqrt{\alpha_k}}{1152}\cdot\frac{1}{x_k^3(n)}< \frac12.
\]
so \(0\le\varepsilon_n^\ast<\tfrac12\).

Then by Lemma~\ref{lem:factor}, a sufficient condition for
\(
\Theta(n)\le 1
\)
is
\begin{equation}\label{eq:sufficient-factor}
1-\Theta_M(n) \ge 16\,\varepsilon_n^\ast.
\end{equation}
According to \eqref{eq:eps-star}, to prove \eqref{eq:sufficient-factor}, it suffices to show
\begin{equation}\label{eq:need-to-show}
1-\Theta_M(n) \ge 24\,\rho_k\,C_k\, x_k(n)\, e^{-\delta_k x_k(n)}.
\end{equation}
Using Lemma \ref{lem:concavity}, it is enough that
\[
\frac{\sqrt{\alpha_k}\,\pi^4}{36\,x_k^3(n+1)} \ge 24\,\rho_k\,C_k\, x_k(n)\, e^{-\delta_k x_k(n)},
\]
i.e.,
\[
e^{\delta_k x(n)}\ge \frac{864\,\rho_k\,C_k}{\pi^4\sqrt{\alpha_k}}\;x_k(n)x_k^3(n+1).
\]
Moreover, since
\[
\frac{x_k^3(n+1)}{x_k^3(n)}
=\Bigl(1+\frac{2\pi^2}{3x_k^2(n)}\Bigr)^{\!3/2}
\le \exp\Bigl(\frac{\pi^2}{x_k^2(n)}\Bigr)
\le \exp\Bigl(\frac{\pi^2\alpha_k}{676}\Bigr)=:\tilde{c}_k,
\]
we have
\[
\frac{864\,\rho_k\,C_k}{\pi^4\sqrt{\alpha_k}}\;x_k(n)x_k^3(n+1) \le \frac{864\,\tilde{c}_k\,\rho_k\,C_k}{\pi^4\sqrt{\alpha_k}}\;x_k^4(n).
\]
Since \(\tilde{c}_k\rho_k<2\), the right-hand side is $\le \dfrac{1728\,C_k}{\pi^4\sqrt{\alpha_k}}\,x_k^4(n)=A_k x_k^4(n)$, which is exactly the hypothesis \eqref{eq:key-ineq}. Hence \eqref{eq:need-to-show} holds, so \eqref{eq:sufficient-factor} holds, and therefore \(\Theta(n)\le 1\), i.e.,
\[
\Delta_k^2(n) \ge \Delta_k(n-1)\,\Delta_k(n+1).
\]
This concludes the proof.
\end{proof}

Having established an available sufficient condition for $\Delta_k^2(n) \ge \Delta_k(n-1)\,\Delta_k(n+1)$, we are now in a position to prove Theorem \ref{thm:main} directly.

{\noindent\it Proof of Theorem \ref{thm:main}.}
By Lemma~\ref{lem:analytic}, it suffices to verify, for a given $n$,
\begin{equation}\label{eq:two-hyps}
\text{(i)}\quad x_k(n-1)\ \ge\ \frac{26}{\sqrt\alpha_k},
\qquad
\text{(ii)}\quad e^{\delta_k\,x_k(n)}\ \ge\ A_k\,x_k(n)^4,
\end{equation}
where $\delta_k=\sqrt{\alpha_k}-\sqrt{\beta_k}>0$ and
\[
A_k=\frac{1728\,C_k}{\pi^4\sqrt{\alpha_k}},
\qquad
C_k=\frac{432\sqrt{2}k^3 e^{\pi\sqrt{k/3}}}{\pi^{\frac{5}{2}}\alpha_k^{\frac{3}{4}}}.
\]

Since
\[
x_k(n-1)=\frac{\pi}{6}\sqrt{\,24(n-1)-(2k+2)\,},
\]
the condition $x_k(n-1)\ge 26/\sqrt\alpha_k$ is equivalent to
\[
24(n-1)-(2k+2)\ \ge\ \frac{26^2\cdot 36}{\pi^2\alpha_k}
\quad\Longleftrightarrow\quad
n\ \ge\ 1+\frac{k+1}{12}+\frac{1014}{\alpha_k\pi^2}.
\]
Hence, for any integer $n\ge \left\lceil \dfrac{1014}{\alpha_k\pi^2}+\dfrac{k+1}{12}+1\right\rceil$,
\eqref{eq:two-hyps}(i) holds.

Let $F(n):=\dfrac{e^{\delta_k x_k(n)}}{x_k^4(n)}$. Since $x_k'(n)=\dfrac{\pi^2}{3x_k(n)}>0$,
\[
\frac{d}{dn}\log \left(F(n)\right)=\delta_k x_k'(n)-\frac{4x_k'(n)}{x_k(n)}
=\frac{\pi^2}{3x_k(n)}\Big(\delta_k-\frac{4}{x_k(n)}\Big).
\]
Hence $F(n)$ is strictly increasing whenever $x_k(n)\ge 4/\delta_k$. At the threshold
\[
n_0(k):=8\,k^{3}+\frac{k+1}{12},
\]
we have
\begin{equation}\label{eq:x0}
x_k\left(n_0(k)\right)=\frac{\pi}{6}\sqrt{24n_0(k)-(2k+2)}
= \frac{\pi}{6}\sqrt{24\Big(n_0(k)-\frac{k+1}{12}\Big)}
=\frac{4\pi}{\sqrt{3}}\,k^{3/2},
\end{equation}
while
\[
\frac{4}{\delta_k} \le \frac{4(2k+1)}{6}\sqrt{\frac{5}{2}}
= \frac{2\sqrt{10}}{3}\,k+\frac{\sqrt{10}}{3}
\]
since
\begin{equation}\label{bound-deltak}
\delta_k=\sqrt{\alpha_k}-\sqrt{\beta_k}=\frac{\alpha_k-\beta_k}{\sqrt{\alpha_k}+\sqrt{\beta_k}}
=\frac{12}{(2k+1)\,(\sqrt{\alpha_k}+\sqrt{\beta_k})}
\ \ge\ \frac{6}{2k+1}\sqrt{\frac{2}{5}}.
\end{equation}
Thus for $k\ge3$,
\[
x_k(n_0) \ge \frac{4\pi}{\sqrt{3}}\,k^{3/2} \ge 4\pi\,k\ >\ \frac{2\sqrt{10}}{3}\,k+\frac{\sqrt{10}}{3}\ \ge\ \frac{4}{\delta_k},
\]
and consequently $F(n)$ is increasing for all $n\ge n_0(k)$.

Combining with the displayed formula for $A_k$ and the elementary bound
$\alpha_k^{-5/4}\le1$, we get
\begin{equation}\label{eq:Ak-rough}
A_k\ \le\ \widetilde{A}_k:=\frac{1728}{\pi^{\frac{13}{2}}}\,\cdot 432\sqrt{2}\;k^{3}\,e^{\pi\sqrt{k/3}}.
\end{equation}
When $n\ge n_0(k)$, the monotonicity of $F(n)$ and \eqref{eq:Ak-rough} give that a sufficient condition for \eqref{eq:two-hyps}(ii) is
\begin{equation}\label{eq:key-ineq-sufficient}
\delta_k\,x_k(n_0(k)) \ge \log \widetilde{A}_k+4\log\left( x_k(n_0(k))\right).
\end{equation}

Applying \eqref{eq:x0} and \eqref{bound-deltak}, for $k\ge 3$,
we obtain the explicit uniform lower bound
\begin{equation*}
\delta_k\,x_k\left(n_0(k)\right) \ge 24\pi\sqrt{\frac{2}{15}}\cdot\frac{k^{3/2}}{2k+1}
 \ge\frac{72\pi}{7}\sqrt{\frac{2k}{15}}.
\end{equation*}
On the other hand, from \eqref{eq:x0} and \eqref{eq:Ak-rough},
\[
\log \widetilde{A}_k+4\log \left(x_k\left(n_0(k)\right)\right)
=\pi\sqrt{k/3}+9\log k + \mathrm{const},
\]
where
\[
\mathrm{const}
:=\log\!\Big(\frac{1728}{\pi^{\frac{13}{2}}}\,\cdot 432\sqrt{2}\Big)
+4\log\!\Big(\frac{4\pi}{\sqrt{3}}\Big)
<14.4.
\]
Consequently, \eqref{eq:key-ineq-sufficient} follows from the inequality
\[
\left(\frac{72\pi}{7}\sqrt{\frac{2}{15}}-\frac{\pi}{\sqrt{3}}\right)\sqrt{k}\ \ge\ 9\log k + 14.4,
\]
which holds for all $k\ge 16$.
For $k=3$ and for $5\le k<16$, we verify \eqref{eq:two-hyps}(ii) \emph{at the threshold} $n=n_0(k)$ by a finite computation using the exact constants $C_k$ (thus the exact $A_k$). By the monotonicity of $F(n)$ established above, this implies that \eqref{eq:two-hyps}(ii) then holds for all $n\ge n_0(k)$ and each such $k$.

For $k=4$, condition \eqref{eq:two-hyps}(ii) at
\(n\ge 526\) is verified by a computation
using the exact constants \(C_k\) (hence \(A_k\)).

Note that for $k\ge3$,
\[\max\left\{\left\lceil
8k^{3}+\frac{k+1}{12}\right\rceil,
\left\lceil\frac{1014}{\alpha_k\pi^2}+\frac{k+1}{12}+1
\right\rceil\right\}=\left\lceil
8k^{3}+\frac{k+1}{12}\right\rceil.
\]
Therefore, for all $k\ge3$ and $n\ge \max\left\{\left\lceil
8k^{3}+\frac{k+1}{12}\right\rceil,526\right\}$,
the two hypotheses \eqref{eq:two-hyps} of Lemma~\ref{lem:analytic} are satisfied.
Therefore $\Delta_k^2(n)\ge \Delta_k(n-1)\Delta_k(n+1)$, which completes the proof.
\qed

 \vspace{0.5cm}
 \baselineskip 15pt
{\noindent\bf\large{\ Acknowledgements}} \vspace{7pt} \par
The author wishes to thank the referee for valuable suggestions. The author is sincerely grateful to her supervisor, Helen W. J. Zhang, for continuous guidance and support. The author also thanks Professor Kathrin Bringmann for valuable discussions and insightful feedback during a visit to the University of Cologne. The author gratefully acknowledges the financial support of the China Scholarship Council (CSC).

\end{document}